\documentclass{amsart} 

\newtheorem{theorem}{Theorem}[section]
\newtheorem{proposition}{Proposition}[section]

\theoremstyle{definition}

\newtheorem{example}[theorem]{Example}

\newtheorem{remark}[section]{Remark}

\usepackage{amscd}
\usepackage{amsthm}
\usepackage{amsfonts}
\usepackage{amsbsy}
\usepackage{amsmath}
\usepackage{amssymb}
\usepackage{color}	
\usepackage{enumerate}
\usepackage[mathscr]{eucal}
\usepackage{fancyhdr}
\usepackage{mathrsfs}
\usepackage{graphicx}
\usepackage{hyperref}
\hypersetup{colorlinks=false, urlcolor=black}
\usepackage{latexsym}
\usepackage{listings}
\usepackage{psfrag}
\usepackage{setspace}
\usepackage{url}
\usepackage{undertilde}
\usepackage{verbatim}
\usepackage{graphics}
\usepackage[all]{xy}
\newcommand{\Z}{\mathbb{Z}}
\newcommand{\Q}{\mathbb{Q}}

\newcommand{\threematr}[9]{\Bigl( \begin{smallmatrix}{#1}&{#2}&{#3}\\{#4}&{#5}&{#6}\\{#7}&{#8}&{#9}\end{smallmatrix} \Bigr) }
\newcommand{\threevect}[3]{\Bigl( \begin{smallmatrix}{#1}\\{#2}\\{#3}\end{smallmatrix}\Bigr) }

\begin{document}

\title[A cubic generalization of Brahmagupta's identity]{A cubic generalization of Brahmagupta's identity}

\author[Samuel A. Hambleton]{Samuel A. Hambleton}

\address{School of Mathematics and Physics, The University of Queensland, St. Lucia, Queensland, Australia 4072}

\email{sah@maths.uq.edu.au}


\date{Submitted to J. R. M. S. on 27 May, 2016.}


\begin{abstract}
We give an algebraic identity for cubic polynomials which generalizes Brahmagupta's identity and facilitates arithmetic in cubic fields. We also pose a question about a relationship between the elements of a cubic field of fixed trace and fixed norm and rational points of an elliptic curve. 
\end{abstract}

\maketitle

\section{Introduction}

Brahmagupta's identity is an ancient Indian algebraic identity with several applications. The identity is expressed as  
\begin{equation}\label{Brahm}
\left( x_1 x_2 + D y_1 y_2 \right)^2 - D \left( x_1 y_2 + x_2 y_1 \right)^2 = \left( x_1^2 - D y_1^2 \right) \left( x_2^2 - D y_2^2 \right) . 
\end{equation}
When $D = -1$, we obtain a well known result on {\em Pythagorean triples}, triples of positive integers $(x,y,z)$ corresponding to the lengths of the sides of a right triangle so that $x^2 + y^2 = z^2$. By \eqref{Brahm}, if $(x_1, y_1, z_1)$ and $(x_2, y_2, z_2)$ are Pythagorean triples, then so is $(x_3, y_3, z_3)$, where 
\begin{align*}
x_3 & = \left| x_1 x_2 - y_1 y_2 \right| , & y_3 & = x_1 y_2 + x_2 y_1 , & z_3 & = z_1 z_2 
\end{align*}
and $| \cdot |$ denotes the absolute value. 

Similarly, if $(x_1, y_1)$ and $(x_2, y_2)$ satisfy the Pell equation 
\begin{equation}\label{Pell}
x^2 - D y^2 = 1 , 
\end{equation}
where $D$ is an integer, then by \eqref{Brahm} so does $(x_3, y_3)$, where  
\begin{align}\label{grplaw}
x_3 & = x_1 x_2 + D y_1 y_2 , & y_3 & = x_1 y_2 + x_2 y_1 . 
\end{align}
This shows that if the Diophantine equation \eqref{Pell} has a solution other than $(\pm 1, 0)$, then more can be obtained. When $D > 0$ is not a square, it is known that \eqref{Pell} has infinitely many solutions.

Possibly the first known method for solving \eqref{Pell}, the chakravala method due to Bhaskara II, uses \eqref{Brahm}. See \cite{Duttaone,Duttatwo,Duttathree} for a description and history of the chakravala method. See \cite{JacWill} for other methods of solving \eqref{Pell}.  

Consider the matrix given by
\begin{equation}\label{Brahmatrix}
M_j = \left(
\begin{array}{cc}
 x_j & D y_j \\
 y_j & x_j \\
\end{array}
\right) . 
\end{equation}
The transpose of the matrix is called the Brahmagupta matrix; see \cite{Suryanarayan}. Brahmagupta's identity \eqref{Brahm} is equivalent to the matrix identity $M_1 M_2 = M_3$. Taking the determinants of the $M_j$, for $j = 1, 2, 3$, gives \eqref{Brahm}, and the left column of $M_3$ is the vector with entries given by \eqref{grplaw}. 

The purpose of this article is to introduce an analogous $3 \times 3$ matrix such that matrix multiplication gives a similar identity to \eqref{Brahm} for a ternary cubic form and gives \eqref{Brahm} when we set specific integer values to the letters in this $3 \times 3$ matrix. Before we introduce this matrix we will consider Brahmagupta's identity in the context of the quadratic field $K = \Q (\sqrt{D})$, where $D$ is the discriminant of $K$. 

Let $\mathcal{O}_K = \Z [\omega ]$, where $\omega = \frac{s + \sqrt{D}}{2}$ is the ring of integers of $K$, and $s$, equal to $0$ or $1$, is the remainder of $D$ modulo $4$. Then for every $\alpha \in \mathcal{O}_K$, there are integers $u, y$ such that $\alpha = u + y \omega $. We can express $\alpha$ as $\alpha = \frac{x + y \sqrt{D}}{2}$, where $x = 2 u + s y$ is the trace of $\alpha $. Taking the norm of $\alpha $ gives 
\begin{equation}\label{intquadr}
N( \alpha ) = \left( \frac{x + y \sqrt{D}}{2} \right) \left(  \frac{x - y \sqrt{D}}{2} \right) = \frac{1}{4}\left( x^2 - D y^2 \right) .
\end{equation}
We know that the norms of algebraic numbers are multiplicative and consequently if $\alpha , \beta \in K$, then $$N(\alpha ) N(\beta ) = N(\alpha \beta ) .$$ When we write $\alpha , \beta $ in the form of \eqref{intquadr}, and multiply by $4$, we obtain \eqref{Brahm}. In fact, there is a bijection $\phi$ between the ring $\mathcal{O}_K$ and the integer solutions $(x, y, n)$ to the equation $x^2 - D y^2 = 4 n$, where $D$ is the discriminant of a quadratic field, given by
\begin{align*}
\phi & : \mathcal{O}_K \longrightarrow \left\{ (x, y, n) \in \Z^3 \ : \ x^2 - D y^2 = 4 n \right\} , & \phi & : \alpha = \frac{x + y \sqrt{D}}{2} \longmapsto \left( x, y, N(\alpha ) \right) ;
\end{align*}  
see \cite{Lemm}. As a consequence of this bijection we are able to carry out basic arithmetic in $\mathcal{O}_K$ and $K$ using $2 \times 2$ matrices with integer and rational entries respectively. If instead of \eqref{Brahmatrix} we use 
\begin{equation}\label{Brahmtwo}
N_j = \left(
\begin{array}{cc}
 u_j & m y_j \\
 y_j & u_j + s y \\
\end{array}
\right), 
\end{equation} 
where $m = \frac{D - s}{4} $, we have the convenient correspondence $\alpha_j = u_j + y_j \omega \longleftrightarrow N_j$ with several useful properties. The trace of $\alpha_j$ is the trace of $N_j$, the norm of $\alpha_j$ is the determinant of $N_j$, $\alpha_j \in \mathcal{O}_K$ if and only if $N_j$ has integer entries, $\alpha_j \in K$ if any only if $N_j$ has rational entries, $\alpha_j^{-1}$ can be calculated by finding the inverse of $N_j$ and considering the entries in the left column, and the eigenvalues of $N_j$ are $\frac{x_j \pm y_j \sqrt{D}}{2}$.    

Gauss \cite[Art. 234]{Gauss} generalized Brahmagupta's identity for the composition of binary quadratic forms, given by $Q(x, y) = A x^2 + B x y + C y^2$, where $A, B, C \in \Z$ satisfying $\gcd (A, B, C) = 1$, $B^2 - 4 A C = D$ is the discriminant of a quadratic field, and $x, y$ are indeterminants. The binary quadratic form $Q(x, y)$ is abbreviated $Q = (A, B, C)$. Gauss' composition law on the $\text{GL}_2(\Z )$ classes of binary quadratic forms is known to be equivalent to the ideal class group of the quadratic field $K$. The way that Gauss introduced composition is known as a bilinear transformation,
\begin{equation}\label{Gauss}
Q_3(u_3 , y_3) = Q_1(u_1 , y_1) Q_2(u_2 , y_2) ,  
\end{equation}
where $Q_j (u_j, y_j) = (A_j , B_j, C_j)$, for $j = 1, 2, 3$. Since we can always find $\beta_j \in \Z$ such that $B_j = 2 \beta_j + s$, and $C_j = \frac{\beta_j^2 + s \beta_j - m}{A_j}$, we have 
\begin{eqnarray*}
u_3 & = & e u_1 u_2 + \frac{e}{A_2} \left( \beta_2 - \beta_3 \right) u_1 y_2 + \frac{e}{A_1} \left( \beta_1 - \beta_3 \right) u_2 y_1 + \frac{e}{A_1 A_2} \left( \beta^{\times } - \beta_3 \beta^{+} \right) y_1 y_2 , \\
y_3 & = & \frac{A_1}{e} u_1 y_2 + \frac{A_2}{e} u_2 y_1 + \frac{\beta^{+}}{e} y_1 y_2 , \\
\beta^{+} & = & \beta_1 + \beta_2 + s , \\
\beta^{\times } & = & \beta_1 \beta_2 + m , \\
e & = & \gcd \left( A_1 , A_2 , \beta^{+} \right) , \\
A_3 & = & \frac{A_1 A_2}{e^2} , 
\end{eqnarray*}
and $\beta_3$ is the least non-negative integer satisfying $\beta_3 \equiv b_3 \pmod{A_3}$, where $p,q,r$ are integers satisfying $e = A_1 p + A_2 q + \beta^{+} r$, and $b_3 = \frac{1}{e} \left( A_1 \beta_2 p + A_2 \beta_1 q + \beta^{\times r} \right)$. It is relatively easy to show that the coefficients of $u_1 u_2$, $u_1 y_2$, $u_2 y_1 $, and $y_1 y_2$ in $u_3 $ and $y_3$ are integers. When $A_1 = A_2 = 1$, $B_1 = B_2 = s$, and $C_1 = C_2 = - m$, we have $A_3 = 1$, $B_3 = s$, and $C_3 = -m$. Letting $u = \frac{x - s y}{2}$ and multiplying \eqref{Gauss} by $4$ again gives \eqref{Brahm}. Thus Gauss composition may be considered as a quadratic generalization of Brahmagupta's identity. 

\section{A cubic identity}

Now we develop the analogous situation, for cubic fields and $3 \times 3$ matrices, of the relationship between elements of quadratic fields and \eqref{Brahmtwo}. A binary cubic form $\mathcal{C}$ is a form given by 
\begin{equation}\label{bcf}
\mathcal{C}(x, y) = a x^3 + b x^2 y + c x y^2 + d y^3 ,
\end{equation}
where $a,b,c,d \in \Z$, and $\mathcal{C}$ is irreducible over $\Q [x, y]$. Belabas \cite{Belabas} showed that there is a fast algorithm for compiling tables of binary cubic forms $\mathcal{C}$ whose discriminant 
\begin{equation}\label{disc}
\Delta = b^2 c^2 + 18 a b c d - 4 a c^3 - 4 b^3 d - 27 a^2 d^2
\end{equation}
is equal to that of the field $K = \Q (\zeta )$, where $\mathcal{C}(\zeta , 1) = 0$, and $\mathcal{C}$ belongs to the image of the Davenport-Heilbronn map \cite{DavHeil}. This map gives a bijection between the conjugacy class of the cubic field $K$ and the $\text{GL}_2(\Z )$ class of the binary cubic forms of the same discriminant as the field $K$ belonging to the image of such a map, where the conjugacy class of the cubic field $L = \Q ( \alpha )$ is mapped to the $\text{GL}_2(\Z )$ class of the binary cubic form $$\mathcal{C}_{L} = \frac{1}{\text{disc}(L)} N \left( (\omega_2 - \omega_2' ) x - (\omega_3 - \omega_3' ) y \right) ,$$ $\{ 1, \omega_2 , \omega_3 \}$ is an integral basis of $L$, $\text{disc}(L)$ is the discriminant of $L$, and $N$ is the norm. It was also shown \cite{Belabas} that we may easily give the integral basis of the ring or integers $\mathcal{O}_K$ of $K$ in terms of the coefficients $a, b$ of $\mathcal{C}$ of \eqref{bcf} as $$B = \{ 1, a \zeta , a \zeta^2 + b \zeta \} ,$$ so that all $\alpha \in \mathcal{O}_K$ be be uniquely represented as 
\begin{equation}\label{alphacubic}
\alpha_j = u_j + x_j \left( a \zeta \right) + y_j \left( a \zeta^2 + b \zeta \right) , 
\end{equation}
where $u_j, x_j, y_j \in \Z$. We we call the binary cubic form $\mathcal{C}$ belonging to the image of the Davenport-Heilbronn map of $K$ the {\em canonical binary cubic form of the cubic field $K$}.

Consider the matrix 
\begin{equation}\label{Samatrix}
N_j = \threematr{u_j}{- a d y_j}{- a d x_j - b d y_j}{x_j}{u_j - b x_j - c y_j}{- c x_j - d y_j}{y_j}{a x_j}{u_j - c y_j} .
\end{equation}
The matrix given by \eqref{Samatrix} corresponds to $\alpha_j \in K$ satisfying \eqref{alphacubic}, when $\mathcal{C} = (a, b, c, d)$ is the canonical binary cubic form of the cubic field $K$. Under the correspondence between $\alpha_j$ and $N_j$, the trace of $\alpha_j$ is equal to the trace of $N_j$, the norm of $\alpha_j$ is equal to the determinant of $N_j$, the product $\alpha_1 \alpha_2 = \alpha_2 \alpha_1$ of elements of $K$ corresponds to the matrix product $N_1 N_2 = N_2 N_1$, the sum $\alpha_1 + \alpha_2$ corresponds to the sum $N_1 + N_2$, and the multiplicative inverse $\alpha_j^{-1}$ corresponds to the matrix inverse $N_j^{-1}$. Thus the correspondence between $\alpha_j $ and $N_j$ facilitates performing arithmetic in the cubic field $K$ since we can do this using matrices. It is easy to show that these claims are true once we consider the following result, which shows that multiplication of matrices of the form \eqref{Samatrix}, whether or not $(a,b,c,d)$ is a binary cubic form, gives a matrix of the same form. We also note that we can choose $a,b,c,d$ so that this generalizes Brahmagupta's identity, Gauss' bilinear transformation when the binary quadratic forms satisfy $Q_1 = Q_2 = (1, B, C)$, and some known identities on cubic polynomials. 

\begin{proposition}\label{propSam}
Let $N_j$ be given by \eqref{Samatrix}, where $a,b,c,d$ are fixed indeterminants and $u_j, x_j, y_j$ are indeterminants for $j = 1, 2, 3$. Then the matrix product $N_3 = N_1 N_2$ is commutative and of the form \eqref{Samatrix}. We obtain Brahmagupta's identity by choosing $a = 0$, $b = 1$, $c = 0$, $d = - D$, and taking the determinants of the $N_j$. Choosing $a = 0$, $b = C$, $c = -B$, and $d = 1$ gives the special case of Gauss' bilinear transformation for equal principal forms $Q_1 = Q_2 = (1, B, C)$. 
\end{proposition}

\begin{proof}
We begin by factorizing the matrix $N_j$ as the product $N_j = S U_j^T$, where $T$ denotes the matrix transpose and 
\begin{align*}
S & = \left(
\begin{array}{cccccc}
 1 & 0 & 0 & 0 & -a d & -b d \\
 0 & 1 & 0 & -b & -c & -d \\
 0 & 0 & 1 & a & 0 & -c \\
\end{array}
\right) , & U_j & = \left(
\begin{array}{cccccc}
 u_j & x_j & y_j & 0 & 0 & 0 \\
 0 & u_j & 0 & x_j & y_j & 0 \\
 0 & 0 & u_j & 0 & x_j & y_j \\
\end{array}
\right) .
\end{align*}
Let 
\begin{eqnarray*}
u_3 & = & u_1 u_2 - a d x_2 y_1 - a d x_1 y_2 - b d y_1 y_2 , \\
x_3 & = & u_1 x_2 + u_2 x_1 - b x_2 x_1 - c x_1 y_2 - c x_2 y_1 - d y_1 y_2 , \\
y_3 & = & u_1 y_2 + u_2 y_1 + a x_1 x_2 - c y_1 y_2 .
\end{eqnarray*}
We must show that $S \left( U_3^T - U_1^T S U_2^T \right) = [0]_{3 \times 3}$, the $3 \times 3$ matrix with zero entries. Expanding $U_3^T - U_1^T S U_2^T $ gives 
\tiny
\begin{equation*}
\left(
\begin{array}{ccc}
 -a d x_2 y_1-b d y_1 y_2 -a d x_1 y_2 & a d u_1 y_2 & a d u_1 x_2 + b d u_1 y_2 \\
 -b x_1 x_2 - c x_2 y_1 - c x_1 y_2 - d y_1 y_2 & b u_1 x_2 - a d x_2 y_1 + c u_1 y_2 - b d y_1 y_2 & c u_1 x_2 + d u_1 y_2 + a d x_1 x_2 + b d x_1 y_2 \\
 a x_1 x_2-c y_1 y_2 & -a u_1 x_2 + a d y_1 y_2 & c u_1 y_2 - a d x_1 y_2 \\
 -x_1 x_2 & u_1 x_2 - c x_2 y_1 - d y_1 y_2 & c x_1 x_2 + d x_1 y_2 \\
 -x_1 y_2 - x_2 y_1 & u_1 y_2 + b x_2 y_1 & u_1 x_2 - b x_1 x_2 \\
 -y_1 y_2 & -a x_2 y_1 & u_1 y_2 + a x_1 x_2 \\
\end{array}
\right) .
\end{equation*}
\normalsize
Multiplying on the left by $S$ produces $[0]_{3 \times 3}$. Since swapping $u_1$ and $u_2$, $x_1$ and $x_2$, $y_1$ and $y_2$ does not change $u_3$, $x_3$, $y_3$, it follows that matrix multiplication of the $N_j$ is commutative. 

To show that matrix multiplication generalizes Brahmagupta's identity, let $a = 0$, $b = 1$, $c = 0$, $d = - D$. Then for $j = 1, 2, 3$, 
\begin{equation}\label{SamBrahm}
N_j = \threematr{u_j}{0}{ D y_j}{x_j}{u_j - x_j }{ D y_j}{y_j}{0}{u_j } ,
\end{equation}   
and expanding the matrix $N_3 = N_1 N_2 $ gives \small $$N_3 = \left(
\begin{array}{ccc}
 u_1 u_2 + D y_1 y_2 & 0 & D (u_1 y_2 + u_2 y_1 ) \\
 u_2 x_1-x_2 x_1+u_1 x_2+D y_1 y_2 & u_1 u_2 - u_2 x_1 - u_1 x_2 + x_1 x_2 & D (u_1 y_2 + u_2 y_1 ) \\
 u_2 y_1+u_1 y_2 & 0 & u_1 u_2+D y_1 y_2 \\
\end{array}
\right) .$$ \normalsize
The determinants of $N_1$ and $N_2$ are $(u_1 - x_1) \left( u_1^2-D y_1^2 \right)$ and $(u_2 - x_2) \left( u_2^2-D y_2^2 \right)$. The determinant of $N_3$ is the product of the determinants of $N_1$ and $N_2$. Assuming $u_1 \not= x_1$ and $u_2 \not= x_2$, we obtain Brahmagupta's identity.  

To prove the claim about the special case of Gauss' bilinear transformation, let $a = 0$, $b = C$, $c = -B$, and $d = 1$. The determinant of $N_j$ is equal to $$ \left( u_j - C x_j + B y_j \right) \left( u_j^2 + B u_j y_j + C y_j^2 \right) .$$ Letting 
\begin{eqnarray*}
u_3 & = & \left(  u_1 u_2 - C y_1 y_2 \right) , \\
x_3 & = & \left(  u_1 x_2 + u_2 x_1 - C x_2 x_1 + B x_1 y_2 + B x_2 y_1 - y_1 y_2 \right) , \\
y_3 & = & \left(  u_1 y_2 + u_2 y_1 + B y_1 y_2 \right) .
\end{eqnarray*}
We have $$u_3 - C x_3 + B y_3 =  \left( u_1 - C x_1 + B y_1 \right) \left( u_2 - C x_2 + B y_2 \right) .$$ Canceling this from the determinant of $N_3 = N_1 N_2$ gives $$Q (u_3, y_3)  = Q (u_1, y_1) Q (u_2, y_2) .$$
\end{proof}

Other choices of $a$, $b$, $c$, and $d$ give identities on well known forms. For example, choosing $a = 1$, $b = 0$, $c = 0$, and $d = - n$, and taking the determinant of $N_j$ gives the form $u^3 + n x^3 + n^2 y^3 - 3 n u x y$ studied by Lagrange and Carmichael \cite{Carmichael}. Consequently, it it easy to show that if there is one solution $(u, x, y) \not= (1, 0, 0)$ to the Diophantine equation $$u^3 + n x^3 + n^2 y^3 - 3 n u x y = 1,$$ then more may be generated via matrix multiplication of the $N_j$. 
 
\begin{proposition}
Let $K = \Q(\zeta )$ be a cubic field, where $\zeta$ is a root of the polynomial $\mathcal{C}(x, 1)$, $\mathcal{C} = (a, b, c, d)$ is the canonical binary cubic form of the cubic field $K$, and let $\mathcal{O}_K $ be the ring of integers of $K$ generated by $\{ 1 , \rho , \omega \}$, where $\rho = a \zeta $, $\omega = a \zeta^2 + b \zeta $. Let 
\begin{align*}
M_{\Z } & = \{ N_j = [a_{ij}]_{3 \times 3} \ : \ a_{ij} \in \Z \} . & M_{\Q } & = \{ N_j = [a_{ij}]_{3 \times 3} \ : \ a_{ij} \in \Q \} . 
\end{align*}
There is a ring isomorphism $\phi$ and a field isomorphism $\psi$ given by
\begin{align*}
\phi & : \mathcal{O}_K \longrightarrow M_{\Z } , & \phi & : \alpha_j \longmapsto N_j , \\
\psi & : K \longrightarrow M_{\Q } , & \psi & : \alpha_j \longmapsto N_j ,
\end{align*}
where $\alpha_j = u_j + x_j \rho + y_j \omega $ and $N_j$ is defined by \eqref{Samatrix}. The trace of $\alpha_j$ is equal to the trace of $N_j$ and the norm of $\alpha_j$ is equal to the determinant of $N_j$.
\end{proposition}

\begin{proof}
First we show that $\phi $ is a bijection. Let $\alpha_1 , \alpha_2 \in \mathcal{O}_K$ and assume $\phi \left( \alpha_1 \right) = \phi \left( \alpha_2 \right)$. Then the left columns of the matrices $\phi \left( \alpha_j \right)$, for $j = 1, 2$, must be equal. It follows that $u_1 = u_2$, $x_1 = x_2$, $y_1 = y_2$ and we must have $\alpha_1 = \alpha_2$. Therefore $\phi$ is injective. To show that $\phi$ is surjective, Let $N_j \in M$. We have $\phi \left( u_j + x_j \rho + y_j \omega \right) = N_j$ and $\alpha_j = u_j + x_j \rho + y_j \omega \in \mathcal{O}_K$. It follows that $\phi$ is bijective. By transport of the ring structure of $\mathcal{O}_K$ onto $M$, we obtain a ring isomorphism. Similarly $\psi $ is a field isomorphism. Under these maps $0$ and $1$ map to the zero matrix and the identity matrix respectively. 

To prove the claims about the trace and norm of $\alpha_j$, let $\tau $ be an embedding of $K$. The trace and norm are defined 
\begin{align*}
t_j & = \alpha_j + \tau \left( \alpha_j \right) + \tau  \left( \tau \left( \alpha_j \right) \right) , & n_j & = \alpha_j \tau \left( \alpha_j \right) \tau \left( \tau \left( \alpha_j \right) \right) ,
\end{align*}   
where without loss of generality we assume $\tau : (\zeta_1 , \zeta_2 , \zeta_3 ) \longmapsto (\zeta_2 , \zeta_3 , \zeta_1 ) $ and we define $\zeta_1 = \zeta$, and let $\zeta_2$, $\zeta_3 $ be the other two roots of $\mathcal{C}(x, 1)$. By considering the identities 
\small
\begin{align*}
\zeta_1 + \zeta_2 + \zeta_3 & = \frac{-b}{a} , & \zeta_1 \zeta_2 + \zeta_1 \zeta_3 + \zeta_2 \zeta_3 & = \frac{c}{a} , \\
 \zeta_1^2 + \zeta_2^2 + \zeta_3^2 & = \frac{b^2 - 2 a c}{a^2} , & \zeta_1^3 + \zeta_2^3 + \zeta_3^3 & = \frac{- b^3 + 3 a b c - 3 a^2 d}{a^3} , \\
\zeta_1^2 \zeta_2^2 + \zeta_1^2 \zeta_3^2 + \zeta_2^2 \zeta_3^2 & = \frac{c^2 - 2 b d}{a^2} , & \zeta_1 \zeta_2 \zeta_3 & = \frac{-d}{a} , \\
 &  & t_j & = 3 u_j - b x_j - 2 c y_j ,
\end{align*}
and 
\begin{eqnarray*}
\left( \zeta_2 + \zeta_3 \right) \zeta_1^2 + \left( \zeta_1 + \zeta_3 \right) \zeta_2^2 + \left( \zeta_1 + \zeta_2 \right) \zeta_3^2 = \frac{-b d + 3 a d}{a^2} ,
\end{eqnarray*}
\normalsize
it is easy although somewhat tedious to show that the traces of $\alpha_j$ and $N_j$ coincide and the norm of $\alpha_j$ coincides with the determinant of $N_j$.
\end{proof}

We note that there is a concise way to represent the determinant of $N_j$, \eqref{Samatrix}. Delone and Faddeev \cite[pp. 130]{DeloneFaddeev} considered various triangular forms, ternary cubic forms. The determinant of $N_j$ is a ternary cubic form, not a Cayley nor a Dirichlet form but if we replace $u \longrightarrow u + c y$, $x \longrightarrow - x$, and then take the determinant, we get the Dirichlet form. When we change to variables $t, x, y$, where $t$ is trace of $N_j$ and $n$ is the determinant, instead of $u, x, y$, we obtain the equation 
\begin{equation}\label{myfaveq}
t^3 - 3 t \mathcal{Q}(x, y) + \mathcal{F}(x, y) = 27 n ,
\end{equation}
where $\mathcal{Q}$ and $\mathcal{F}$ are respectively the Hessian and Jacobian covariant binary forms, of the binary cubic $\mathcal{C} = (a, b, c, d)$, defined using partial derivatives
\begin{align*}
H(\mathcal{C}) & = \frac{1}{2} \left(
\begin{array}{cc}
 \mathcal{C}_{xx} & \mathcal{C}_{xy} \\
 \mathcal{C}_{yx} & \mathcal{C}_{yy} \\
\end{array}
\right) , & J_{(\mathcal{Q}, \mathcal{C})} & = \left(
\begin{array}{cc}
 \mathcal{Q}_x & \mathcal{Q}_y \\
 \mathcal{C}_x & \mathcal{C}_y \\
\end{array}
\right) , \\
\mathcal{Q}(x, y) & = - \det ( H(\mathcal{C}) ) , & \mathcal{F}(x, y) & = - \det (J_{(\mathcal{Q}, \mathcal{C})}) ,
\end{align*}
and satisfy the syzygy 
\begin{equation}\label{syzygy}
\mathcal{F}^2 + 27 \Delta \mathcal{C}^2 = 4 \mathcal{Q}^3 ;
\end{equation}
see \cite{Hilbert}, used by Mordell \cite{Mordell} to show that the elliptic curve $y^2 = x^3 + k$ has finitely many integer solutions. Note that the Hessian binary quadratic form is simply 
\begin{equation}\label{bqf}
\mathcal{Q} = \left( b^2 - 3 a c,  b c - 9 a d,  c^2 - 3 b d \right) 
\end{equation}
and while $\mathcal{F}$ is not defined by \eqref{syzygy}, it can be obtained from  \eqref{disc}, \eqref{syzygy} and \eqref{bqf}. See \cite{Bhar} for the relationship between binary cubic forms and the ideal class group of a cubic field. 

It can be shown that when $\mathcal{C}$ is the canonical binary cubic form of a cubic field $K$ and $n = 1$, then as a Diophantine equation \eqref{myfaveq} has a group law such that the group of integer points $(t, x, y)$ is isomorphic to the group of units of $K$ of norm $1$ and this group law may be evaluated by matrix multiplication of the $N_j$. Similarly we may restrict attention to the set of all algebraic numbers of a cubic field having trace $0$. This set will form a group of trace $0$ elements under addition. It is natural to ask whether there is a group of elements of a cubic field having both trace $0$ and norm $1$. The answer is affirmative. There is a relationship between certain rational points $(x, y)$ of \eqref{myfaveq} and certain elliptic curves but the group law cannot be viewed in terms of matrix multiplication or addition of the $N_j$. It shows that rational points of some elliptic curves may be thought of as corresponding to elements of a cubic field with a specific trace and specific norm. See \cite{ManasaShankar} for a discussion of a relationship between elliptic curves and some quadratic fields.

\begin{remark}\label{elliptic}
Let $\mathcal{C} = (a, b, c, d)$ be the canonical binary cubic form of the cubic field $K$ and let $\mathcal{Q}$ and $\mathcal{F}$ the Hessian and Jacobian covariant binary forms of $\mathcal{C}$ respectively. There is a bijection between those elements of $K$ of fixed trace $t$ and fixed norm $n$ and the rational points $(x, y)$ of the elliptic curve \eqref{myfaveq} provided that the choice of fixed $t$ and $n$ make \eqref{myfaveq} a non-singular curve containing a rational point. When $t^3 = 27 n$, the curve is singular.
\end{remark}

We now discuss the transformation of \eqref{myfaveq} to Weierstrass form following Duif \cite{Duif}. In homogeneous coordinates we have 
\begin{equation}\label{homog}
\Gamma (X,Y,Z) = \mathcal{F}(X, Y) - 3 t Z \mathcal{Q} (X, Y) + \left( t^3 - 27 n \right) Z^3 = 0 , 
\end{equation}
where $x = \frac{X}{Z}$, $y = \frac{Y}{Z}$. Let $P = (P_X : P_Y : P_Z)$ be a point of \eqref{homog} in homogeneous coordinates. Since $\Gamma$ is smooth, the tangent line at $P$ is 
\begin{equation}\label{tangline}
\Gamma_X (P) (X - P_X ) + \Gamma_Y (P) (Y - P_Y ) + \Gamma_Z (P) (Z - P_Z ) = 0 .
\end{equation}
The intersections of this line with \eqref{homog} may either be of multiplicity $3$ at $P$, Case 1, or of multiplicity $2$ at $P$ and at another point $Q$ of multiplicity $1$, Case 2. 

In the first case, we map $P$ to $(0:1:0)$ and the line \eqref{tangline} to the line $Z = 0$. We fix a point $Q \not=P$ on \eqref{tangline}, not on \eqref{homog} and satisfying $\det (M) \not= 0$, where $$M = \threematr{Q_X}{P_Z}{1}{Q_Y}{P_Y}{0}{Q_Z}{P_Z}{0} .$$ We send $Q$ to $(1 : 0: 0)$. The matrix $M$ is invertible since $\det (M) \not= 0$. Letting $\threevect{U}{V}{T} = M^{-1} \threevect{X}{Y}{Z} $, we have $$k U^3 + p U^2 T + q U V T + r V^2 T + s U T^2 + w V T^2 + u T^3 = 0 ,$$ where $k, p, q, r, s, w, u \in \Q$. Dividing by $k$, which must be non-zero, and letting $T = - \frac{k W}{r}$ gives the projective Weierstrass form. We see the affine Weierstrass form when we then replace $W$ by $1$. 

Now assume that we have a point $P$ of \eqref{homog} and \eqref{tangline} intersecting \eqref{homog} with multiplicity $2$. Let $Q$ be the other intersection of \eqref{homog} and \eqref{tangline}. We take a new projective tangent line at $Q$. If this new line intersects \eqref{homog} with multiplicity $3$, we find the Weierstrass form using Case 1 renaming $Q$ as $P$. Otherwise we call the other intersection of \eqref{homog} and our new tangent line the projective point $R$. Since $P$, $Q$, $R$ cannot be collinear, the matrix $$M = \threematr{P_X}{Q_X}{R_X}{P_Y}{Q_Y}{R_Y}{P_Z}{Q_Z}{R_Z}$$ is invertible. Letting $\threevect{U^2}{V T}{U T} = M^{-1} \threevect{X}{Y}{Z} $ and dividing the resulting equation satisfied by $T, U, V$ by $U^2 T$, we obtain an equation for which we may easily obtain the Weierstrass form by then dividing this equation by the coefficient of $U^3$ and then replacing $T$ by $1$.  

\begin{example}
Let $\mathcal{C} = (1, 1, 2, 1)$, the canonical binary cubic form of a cubic field $K$ of discriminant $-23$. We fix $t = 0$ and $n = 1$. The covariant forms of $\mathcal{C}$ are $\mathcal{Q} = (-5,-7, 1)$ and $\mathcal{F} = (-11, 39, 48, 25)$. From a unit of $K$ of norm $1$ and trace $0$ we have the rational point $(-1, 1)$ of the affine elliptic curve $E_1 : \mathcal{F}(x, y) = 27 $. In homogeneous coordinates we have $P = (-1 : 1: 1)$. To find the Weierstrass form of this elliptic curve, we consider the tangent line at $P$, given by $Y = \frac{1}{2}(7 X + 9 Z )$. Substituting into $\mathcal{F}(X, Y) - 27 Z^3$ shows that there is one other intersection at $Q = (-29 : 2 : 23)$ of the tangent line with the projective version of $E_1$. The tangent line at $Q$ is $155 X - 133 Y + 207 Z = 0$ , which gives $R = (-6968 : 27569 : 22931)$ and now we have $$M = \threematr{-1}{-29}{-6968}{1}{2}{27569}{1}{23}{22931} .$$ Using this together with $\mathcal{F}(X, Y) = 27 Z^3$, gives the affine elliptic curve in Weierstrass form,
\begin{equation*}
v^2 - 88012375 v - 4056 u v = 55 u^3 + 1314066 u^2 .
\end{equation*}
\end{example}

It is clear that the $\text{GL}_2(\Z)$ class of binary cubic forms of the canonical binary cubic form $\mathcal{C}$ corresponding to $K$ will give, once we fix rational numbers $t, n$ such that \eqref{myfaveq} is non-singular, a class of isomorphic elliptic curves \eqref{myfaveq}. Remark \ref{elliptic} suggests the following question. For which elliptic curves $E$ can we find a cubic field $K$ and rational numbers $t$ and $n$ such that the canonical binary cubic form $\mathcal{C}$ has covariant forms $\mathcal{Q}$ and $\mathcal{F}$ which give an elliptic curve \eqref{myfaveq} defined over $\Q$ which is isomorphic to $E (\Q )$ ? 

The matrices presented in this article are also useful in understanding Voronoi's algorithm; see \cite{cfg}.

\end{document}